\documentclass[12pt,reqno]{amsart}

\usepackage{upgreek}     
\usepackage{enumitem}
\usepackage{graphicx}
\usepackage{mathrsfs}
\usepackage{amsthm}
\usepackage{amssymb}
\usepackage{amsmath,amsthm,amsfonts}
\usepackage{mathabx}
\usepackage{hhline}
\usepackage{textcomp}
\usepackage{amsmath,amscd}
\usepackage[all,cmtip]{xy}
\usepackage{mathtools}
\usepackage[margin=1in]{geometry}
\usepackage{lipsum}
\usepackage{extarrows}
\usepackage{upgreek}
\usepackage{bbm}
\usepackage[displaymath]{lineno}
\usepackage[singlespacing]{setspace}%
\usepackage[colorlinks=true,breaklinks=true,bookmarks=true,urlcolor=blue,
citecolor=blue,linkcolor=blue,bookmarksopen=false,draft=false]{hyperref}

\providecommand{\U}[1]{\protect\rule{.1in}{.1in}}

\theoremstyle{definition}
\newtheorem{theorem}{Theorem}[section]
\newtheorem{lemma}{Lemma}[section]
\newtheorem{remk}{Remark}[section]

\newtheorem{defn}{Definition}[section]
\newtheorem*{theorem*}{Theorem}
\newtheorem{cor}{Corollary}[section]

\numberwithin{equation}{section}

\newcommand{\abs}[1]{\lvert#1\rvert}

\DeclareMathAlphabet{\mathpzc}{OT1}{pzc}{m}{it}

\newcommand{\ftt}[1] {\mathsf{#1}}

\newcommand{\ps}{\mathrm{({PS}})_c}

\newcommand{\va}{\varphi}
\newcommand{\cs}{continuous }

\newcommand{\dd}{{\tt D}}
\newcommand{\dt}[1]{{\tt d}{#1}}
\newcommand{\fs}[1]{\mathbbm {#1}}
\newcommand{\eu}[1]{\EuScript {#1}}
\newcommand\Set[2]{\{\,#1\mid#2\,\}}

\newcommand*{\medcap}{\mathbin{\scalebox{1.5}{\ensuremath{\cap}}}}
\newcommand*{\medcup}{\mathbin{\scalebox{1.5}{\ensuremath{\cup}}}}
\newcommand{\lc}{\mathsf{L}}

\newcommand{\mt}{\mathbbm {d}}

\newcommand{\set}[1]{\left\{#1\right\}}
\newcommand{\snorm}[2][]{\left\lVert#2\right\rVert_{#1}}

\newcommand{\zero}[1]{\mathbf{0}_{#1}}

\newcommand{\hh}{\mathcal{H}}
\newcommand{\nbd}{neighborhood \,}

\newcommand{\rr}{\mathbb{R}}
\newcommand{\nn}{\mathbb{N}}

\newcommand{\uu}{\mathcal{U}}

\newcommand{\f}{Fr\'{e}chet }

\newcommand{\bl}[1] {\mathbf {#1}}

\newcommand{\bb}{\mathcal{B}}
\DeclareMathAlphabet\EuScript{U}{eus}{m}{n}
\SetMathAlphabet\EuScript{bold}{U}{eus}{b}{n}
\newcommand\opn{\ensuremath{\mathrel{\mathpalette\opncls\circ}}}

\newcommand{\opncls}[2]{
	\ooalign{$#1\subseteq$\cr
		\hidewidth\raisefix{#1}\hbox{$#1{\stylefix{#1}#2}\mkern2mu$}\cr}}
\def\raisefix#1{
	\ifx#1\displaystyle
	\raise.39ex
	\else
	\ifx#1\textstyle
	\raise.39ex
	\else
	\ifx#1\scriptstyle
	\raise.275ex
	\else
	\raise.150ex
	\fi
	\fi
	\fi
}
\def\stylefix#1{
	\ifx#1\displaystyle
	\scriptstyle
	\else
	\ifx#1\textstyle
	\scriptstyle
	\else
	\ifx#1\scriptstyle
	\scriptscriptstyle
	\else
	\scriptscriptstyle
	\fi
	\fi
	\fi
}
\DeclareFontFamily{U}{mathx}{\hyphenchar\font45}
\DeclareFontShape{U}{mathx}{m}{n}{
	<5> <6> <7> <8> <9> <10>
	<10.95> <12> <14.4> <17.28> <20.74> <24.88>
	mathx10
}{}
\newcommand{\fr}{Fr\'{e}chet }

\newcommand{\Cl}[1]{\overline{#1}}

\newcommand{\Int}{\mathop{\mathrm{int}}\nolimits}
\newcommand{\Id}{\mathop\mathrm{Id}\nolimits}

\renewcommand{\emptyset}{\varnothing}

\makeatletter
\@namedef{subjclassname@2020}{%
	\textup{2020} Mathematics Subject Classification}
\makeatother
\begin{document}

\title{Some critical point results for Fr\'{e}chet manifolds}

\author{Kaveh Eftekharinasab}
\address{Topology lab.  Institute of Mathematics of National Academy of Sciences of Ukraine, Tereshchenkivska st. 3, Kyiv, 01601 Ukraine}

\email{kaveh@imath.kiev.ua}





\begin{abstract}
  We prove a so-called linking theorem and some of its corollaries, namely a mountain
  pass theorem and a three critical points theorem for Keller $ C^1$-functional on $ C^1 $-\fr manifolds. 
  Our approach relies on a deformation result which is not implemented by considering the negative pseudo-gradient  flows. Furthermore, for mappings between \fr manifolds we provide a set of sufficient conditions in terms of the Palais-Smale condition that indicates when a local diffeomorphism is a global one. 
\end{abstract}

\maketitle
\section{Introduction}
For Banach and Hiblert manifolds there are two approaches to the critical point theory. One is
based on deformation techniques along the negative gradient flow or a suitable substitute of it, namely the pseudo-gradient flow. The other one relies on various versions of Ekeland's variational principle. At the core of both approaches lie Palais-Smale compactness-type conditions.
However,
as pointed out in \cite{k} these approaches do not work in full extent for more general context of Fr\'{e}chet manifolds.
Since for \fr manifolds cotangent bundles do not
admit smooth manifold structures and consequently the notion of pseudo-gradient vector fields  and  Finsler structures on cotangent bundles make no sense.

In this regard, it was introduced the Palais-Smale condition on  Fr\'{e}chet manifolds by using an auxiliary function to detour the need of a smooth structure on cotangent bundles in \cite{k}. The idea behind the definition is
that on sets where a real-valued functional on a manifold has no critical points and satisfies the proposed Palais-Smale condition, the associated auxiliary function  is negative (Lemma \ref{contra}). In this case, the functional satisfies the hypotheses of the deformation result (Lemma \ref{defor}) which requires that the associated auxiliary function be negative. 
Moreover, by imposing the closedness assumption on mappings and applying the deformation result along with the Palais-Smale condition the Minimax principle (Theorem \ref{minmax}) was obtained.
The closedness assumption is crucial in the Minimax principle, without it the theorem is not valid.

In this paper following the ideas of \cite{k} and applying the mentioned results we develop the critical point theory for \fr manifolds. First, we prove a so-called linking theorem (Theorem \ref{th:link}). Then, we obtain some of its corollaries, 
namely a mountain pass theorem  (Theorem \ref{th:mpt}) and a three critical points theorem (Theorem \ref{co:tcpt}). 
Furthermore, we apply the mountain pass theorem and the Palais-Smale condition to provide a set of sufficient conditions that indicates when a local diffeomorphism is a global one (Theorem \ref{th:gd}). 

\section{
	Preliminaries }
In this section we briefly recall the basic concepts of the theory of Fr\'{e}chet manifolds and establish our notations.  

By $ U \opn \mathsf{T} $ we mean that $ U $
is an open subset of a topological space $ \mathsf{T} $.  If $ \mathsf{S}$ is another topological space, then we denote by $ \eu{C} (\mathsf{T},\mathsf{S}) $ the set of  \cs mappings from  $ \mathsf{T} $ into $ \mathsf{S}$.

We denote by $\fs{F}$ a Fr\'{e}chet space whose topology is defined by a sequence of seminorms $(\snorm[\fs{F},n]{\cdot})_{n \in \nn}$, 
which we can always assume to be increasing (by considering $\max_{k \leq n}\snorm[\fs{F},n]{\cdot}$, if necessary). Moreover, the complete translation-invariant metric
\begin{equation*}\label{dm}
\mt_{\fs{F}} (x,y) \coloneqq \sum _{n \in \nn} \dfrac{1}{2^n}\cdot \dfrac{\snorm[\fs{F},n]{x-y}}{1+ \snorm[\fs{F},n]{x-y}}
\end{equation*}
induces the same topology on $\fs{F}$. 

Define a closed unit semi-ball centered at the zero vector $\zero{\fs{F}}$ of $\fs{F}$ by $$\mathbf{B}^n(\zero{\fs{F}},1) = \{ x \in \fs{F} \mid \, \snorm[\fs{F},n]{x} \leq 1\}$$ for each seminorm $\snorm[\fs{F},n]{\cdot}$.
Let
\begin{equation}\label{eq:b}
\mathbf{B}_{\infty}(\zero{\fs{F}}) = \bigcap_{i=1}^{\infty} \mathbf{B}^i(\zero{\fs{F}},1).
\end{equation}
The set $\mathbf{B}_{\infty}(\zero{\fs{F}})$ is not empty ($\zero{\fs{F}} \in \mathbf{B}_{\infty}(\zero{\fs{F}}))$ and is infinite (because it is convex so by the Kolmogorov theorem it is bounded only in Banach spaces). 

 We recall that  a family $\bb$ of  subsets of $\fs{F}$ that covers $\fs{F}$ is called a bornology on $\fs{F}$ if
\begin{description}
	\item [(B1)]  $\forall A, B \in \bb$ there exists $C \in \bb$ such that $A \medcup B \subset C$,
	\item [(B2)] $\forall B \in \bb$ and $ \forall r \in \rr$ there is a $C \in \bb$ such that $r \cdot B \subset C$.
\end{description}
Throughout the paper we assume that   $\fs{F},\fs{E}$ are Fr\'{e}chet spaces and $\lc(\fs{E},\fs{F})$ is the set of all continuous linear mappings from $\fs{E}$ to $\fs{F}$.  Let $\bb$ a bornology on $\fs{E}$. We define on
 $\lc(\fs{E},\fs{F})$ the $\bb$-topology which is a Hausdorff locally convex topology defined by all seminorms:
\begin{equation*} \label{iq}
\snorm[B,n]{L} \coloneq \sup \Set{\snorm[\fs{F},n]{e}}{  e \in B },
\end{equation*} 
where $ B \in \bb, n \in \nn$. We shall always assume that $ \bb $ contains all compact subsets (for simplicity we refer to such bornologies as compact bornologies).

Let $\varphi: U \opn \fs{E} \to \fs{F}$  be a continuous map and $ \bb $ the compact bornology on $ \fs{E} $. If the directional  derivatives
$$\dd \varphi(x)h = \lim_{ t \to 0} \dfrac{\varphi(x+th)-\varphi(x)}{t}$$
exist for all $x \in U$ and all $ h \in \fs{E} $, and  the induced map  $\dd \varphi(x) : U \to \lc(\fs{E},\fs{F})$ is continuous for all
$x \in U$, then  we say that $ \varphi $ is a Keller's differentiable map of class  $C^1$, see \cite{ke}. Here, $\lc(\fs{E},\fs{F})$
is endowed by the $ \bb $-topology which coincides with the compact-open topology.

 A $C^1$-Fr\'{e}chet manifold $ \fs{M} $ is a Hausdorff second countable manifold modeled on a Fr\'{e}chet space $ \fs{F}$ with an atlas of coordinate 
charts  such that the coordinate transition functions are all
Keller $ C^{1} $-mappings.

If $\varphi: \fs{F} \to \rr$ at $x $ is of class $C^1$,
the derivative of $\varphi$ at $x$, $\varphi'(x)$, is an element of the dual space $\fs{F}'$. The directional derivative of $\varphi$ at $x$ toward $h \in \fs{E}$ is given by 
$$ \dd \varphi(x)h = \langle \varphi'(x),h \rangle,$$ where $ \langle \cdot,\cdot \rangle  $ is duality pairing.
Let $x \in \fs{M}$ and $h \in T_x\fs{M}$. A chart  $(x \in U,\psi)$ induces a canonical map $\psi_{*}$ from $T_x\fs{M}$ onto $\fs{F}$.
Let $\varphi : \fs{M} \to \rr$ be a $C^1$-functional, then
$$
\varphi'(x,h) = \lim_{ t \to 0}\dfrac{\varphi \big( \psi^{-1}( \varphi(x) + t \psi_{*}(x)(h)) \big) - \varphi(x)}{t}.
$$ 
\begin{defn}\label{defni}\cite{sn1}
	Let $\fs{F}$ be a Fr\'{e}chet space, $\ftt{T}$  a topological space and $V = \ftt{T} \times \fs{F}$ the trivial bundle with fiber $\fs{F}$ over $\ftt{T}$. A Finsler
	structure for $V$ is a collection of continuous functions $\snorm[V,n]{\cdot}: V \to \mathbb{R}^+$, $n \in \nn$, such that 
	\begin{description}
		\item[(F1)] For $b \in \ftt{T}$ fixed, $ \snorm[\fs{F},n]{x}^b \coloneq \snorm[V,n]{(b,x)}$ is a collection of seminorms 
		on $\fs{F}$ which gives the topology of $\fs{F}$.
		\item [(F2)]Given $ k >1$ and $x_0 \in \ftt{T}$, there exists a neighborhood $W$ of $x_0$  such that
		\begin{equation*} \label{ine}
		\dfrac{1}{k}\snorm[\fs{F},n]{x}^{x_0} \,\leq \,\snorm[\fs{F},n]{x}^{w}\, \leq k \snorm[\fs{F},n]{x}^{x_0}
		\quad
		\text{for all} \quad  w \in W, n \in \nn, x \in \fs{F}.
		\end{equation*}
	\end{description}
\end{defn}
Suppose $\fs{M}$ is a  Fr\'{e}chet manifold modeled on $\fs{F}$. 
Let $\uppi_\fs{M} : T\fs{M} \rightarrow \fs{M}$ be the tangent bundle and let $\snorm[\fs{M},n]{\cdot}: T\fs{M} \rightarrow \mathbb{R}^+$ be a collection of continuous functions, $n \in \nn$.  We say that
$\{\snorm[\fs{M},n]{\cdot}\}_{n \in \nn} $ is a Finsler structure for 
$T\fs{M}$ if for a given $x \in \fs{M}$ there exists a bundle chart $\psi : U \times \fs{F} \simeq T\fs{M}\mid_U$ with $x \in U$  
such that
$$\{\snorm[V,n]{\cdot} \circ \, \psi^{-1}\}_{n \in \nn} $$
is a Finsler structure for $V = U \times \fs{F}$.

A Fr\'{e}chet Finsler manifold is a Fr\'{e}chet manifold together with a Finsler structure on its tangent bundle. Regular 
(in particular paracompact) manifolds admit Finsler structures.

If $\{ \snorm[\fs{M},n]{\cdot} \}_{n \in \nn}$ is a Finsler structure for $\fs{M}$ then  we can obtain a graded Finsler structure, denoted  by $( \snorm[\fs{M},n]{\cdot} )_{n \in \nn}$, that is $\snorm[\fs{M},i]{\cdot} \leq
 \snorm[\fs{M},i+1]{\cdot}$ for all $i \in \nn$.

We define the length of a $C^1$-curve $\gamma : [a,b] \rightarrow M$ with respect to the $n$-th component by 
\begin{equation*}
L_n(\gamma) = \int_a^b \snorm[\fs{M},n]{\gamma'(t)}^{\gamma(t)}\dt t.
\end{equation*}
The length of a  piecewise path  with respect to the $n$-th component is the sum over the curves constituting to the path. On each connected component of $\fs{M}$, the distance is defined by
\begin{equation*} 
\uprho_n (x,y) = \inf_{\gamma} L_n(\gamma),
\end{equation*}
where infimum is taken over all  piecewise $C^1$-curve connecting $x$ to $y$. Thus,
we obtain an increasing sequence of metrics $\uprho_n(x,y)$ and define the distance $\uprho$ by
\begin{equation}\label{finmetric}
\uprho (x,y) = \displaystyle \sum_{n = 1}^{\infty} \dfrac{1}{2^n} \cdot \dfrac {\uprho_n(x,y)}{ 1+ \uprho_n(x,y)}.
\end{equation}

The distance $\uprho$ defined by~\eqref{finmetric} is a metric for $\fs{M}$ which is bounded by 1. Furthermore, the topology induced by this metric coincides with the original topology of $\fs{M}$ (see~\cite{sn1}). 

We denote by
$ \fs{B}_{\uprho}(x,r) $ an open ball with center $ x $ and radius $ r>0 $. We do not write the metric $ \uprho $
when it does not cause confusion.
\section{Linking Results and Corollaries}
Henceforth we assume that $\fs{M}$ is a connected $C^1$-Fr\'{e}chet manifold modeled on $ \fs{F} $ endowed with a complete Finsler  metric $ \uprho $ \eqref{finmetric}, and  that $\varphi : \fs{M} \to \rr$ is a non-constant $C^1$-functional. 
Let $x\in \fs{M}$, we shall say that $x$ is a critical point of $\varphi$ if $(\varphi \psi^{-1})'(\psi(x)) =0$ for
a chart $(x\in U,\psi)$ and hence for every chart whose domain contains $x$.

Let $( \snorm[\fs{M},n]{\cdot} )_{n \in \nn}$ be a graded Finsler structure on $T\fs{M}$. 
Define a closed unit semi-ball centered at the zero vector $\zero{x}$ of $T_x\fs{M}$ by $$\fs{B}^n(\zero{x},1) = \{ h \in T_x\fs{M} : \, \snorm[\fs{M},n]{h}^x \leq 1\}$$ for each $x \in  \fs{M}$ and each 
$\snorm[\fs{M},n]{h}^x $.
Let $$\fs{B}_{\infty}(\zero{x}) = \bigcap_{n=1}^{\infty} \fs{B}^n(\zero{x},1).$$
The set $\fs{B}_{\infty}(0_x)$ is not empty and  infinite because it can be identified with a convex neighborhood of the zero of the
Fr\'{e}chet space $U \times  \fs{F}$, where $U$ is an open neighborhood of $x$.

Let $\varphi: \fs{M} \to \rr$ be a $C^1$-functional and $x \in \fs{M}$. Define
\begin{equation}
\Phi_{\varphi}(x) = \inf \big \{ \varphi'(x,h) : h \in \fs{B}_{\infty}(\zero{x}) \big\}.
\end{equation}
\begin{defn}[The $ \ps $-condition for \fr manifolds, \cite{k}]
	We say that a $C^1$- functional $\varphi: \fs{M} \to \rr$ satisfies the Palais-Smale condition at a level $c \in 
	\rr$, $\ps$ in short, in a set $A \subset \fs{M}$ if any sequence $( m_i)_{i \in \nn} \subset A$ such that $$\varphi(m_i) \to c  \quad \mathrm{and} \quad
	\Phi_{\varphi}(m_i)  \to 0, $$ has a convergent subsequent.
\end{defn}
For \fr spaces this version of the $ \ps $-condition will become as follows. 

Let $\varphi: \fs{F} \to \rr$ be a $C^1$-functional and $x \in \fs{F}$. Define
\begin{equation}
\Psi_{\varphi}(x) = \inf \big \{ \langle \varphi'(x),h \rangle \mid h \in \mathbf{B}_{\infty}(\zero{\fs{F}}) \big\}.
\end{equation}
\begin{defn}[The $ \ps $-condition for \fr spaces]\label{d:nps}
	We say that a $C^1$- functional $\varphi: \fs{F} \to \rr$ satisfies the Palais-Smale condition at a level $c \in 
	\rr$, $\ps$ in short, in a set $A \subset \fs{F}$ if for any sequence $( m_i)_{i \in \nn} \subset A$ such that $$\varphi(m_i) \to c  \quad \mathrm{and} \quad
	\Psi_{\varphi}(m_i)  \to 0, $$ has a convergent subsequent.
\end{defn}
\begin{remk}
	The above version of the Palais-Smale condition is equivalent to the one which was introduced in \cite{k3}.
	Indeed, since  the set $ \mathbf{B}_{\infty}(\zero{\fs{F}}) $ is absorbing it follows that for any $ x \in \fs{F} $ there is a $ k>0 $ such that $ x = kp $ for some $ p \in \mathbf{B}_{\infty}(\zero{\fs{F}}) $. Thus, $ \Psi_{\varphi}(m_i)  \to 0 $ implies $ \va'(m_i) \to 0$ and vice versa. However,
	for technical reasons in many situations it is more convenient to work with Definition \ref{d:nps}. 
\end{remk}
Let $\varphi : \fs{M} \to \rr$ be a $C^1$-functional. We denote by
$\mathrm{Cr(\varphi)}$ the set of critical points of $\varphi$, and for $c \in \rr$ we define the following sets
$$\mathrm{Cr}(\varphi,c) \coloneq \{ x\in \mathrm{Cr(\varphi)} \mid \varphi(x)=c\},$$
$$\varphi^c \coloneq \{ x \in \fs{M} \mid \varphi(x) \leq c \}.$$
A mapping $\hh \in \eu{C}([0,1] \times \fs{M} \to \fs{M})$ is called a deformation if $\hh (0,x) =x$ for all $x \in \fs{M}$.
Let $C$ be a subset of $\fs{M}$, we say that  $\hh$ is a $C$-invariant for an interval $I \subset [0,1]$ 
if $\hh(t,x) = x$ for all $x\in C$ and all $t \in I$.

A family $\mathcal{F}$ of subset of $\fs{M}$ is said to be deformation invariant if for each $A \in \mathcal{F}$ and each deformation $\hh$
for $\fs{M}$ it follows that $$\hh_1(x) \coloneq \hh(1,x) \in \mathcal{F}.$$

\begin{defn}[linking, cf.~\cite{l}]
	Let $ \mathsf{T} $ be a topological space, $ S_0 \subset S $ and $ C $ be a nonempty sets in $ \mathsf{T} $,
	and  let $ \gamma \in \eu{C}(S_0,\mathsf{T}) $. Consider the class of \cs functions
	$$
	\eu{H} \coloneq \Set{h \in \eu{C}(S,\mathsf{T})}{ h|_{\partial S_0} = \gamma}.
	$$
	We say that the pair $ \set{S_0,S} $ links $ D $ through $ \gamma $ if the following holds:
	\begin{description}
		\item[(L1)] $ \gamma(S_0) \medcap C = \emptyset $,
		\item[(L2)] $ \forall h \in \eu{H} $ we have $ h(S) \medcap C \neq \emptyset. $
	\end{description}
The set $ \set{S_0,S,C} $ is said to be a linking set through $ \gamma $. If $ \gamma = \Id_{S_0} $, then $ \set{S_0,S,C} $
is called simply the linking set.
\end{defn}
We shall need the following results.
\begin{lemma}[Lemma 3.2, \cite{k}]\label{contra}
	If $\varphi : \fs{M} \to \rr$ satisfies the Palais-Smale condition in $A \subset M$ and has no critical point
	in $A$, then there exists $\epsilon>0$ such that $\Phi_{\varphi} (x) < -\epsilon $ for all $x \in A$.
\end{lemma}
\begin{lemma}[\cite{k}, Lemma 3.1]\label{defor}
	Let $\fs{M}$ be a connected $C^1$-Fr\'{e}chet manifold endowed with a complete
Finsler metric $\uprho$. Assume $\varphi: \fs{M} \to \rr$ is a $C^1$-functional. Let $B$ and $A$ be closed disjoint subsets of
$\fs{M}$ and let $A$ be compact. Suppose $k>1$ and $\epsilon>0$ are
such that $\Phi_\varphi(x) < -2\epsilon(1+k^2)$ for all $x \in A$. Then there exist $t_0>0$ and
$B$-invariant deformation $\hh$ for  $ [0,t_0)$ such that
\begin{enumerate} 
	\item  $\uprho (\mathcal{H} (t,x),x) \leq kt \quad \forall x \in \fs{M}$,
	\item  $
	\varphi(\hh(t,x)) -\varphi (x) \leq - 2\epsilon(1+k^2) t \quad \forall x \in \fs{M}.
	$\label{defo}
\end{enumerate}
\end{lemma}
\begin{cor}[\cite{k}, Corollary 3.5]\label{titi}
	Let $\fs{M}$ be a connected $C^1$-infinite dimensional Fr\'{e}chet manifold, $\varphi: \fs{M} \to \rr$  a  closed non-constant $C^1$-functional. Suppose $\varphi$ satisfies the Palais-Smale condition at all levels. 
	\begin{enumerate}[label=\textbf{(DF\arabic*)},ref=DF\arabic*]
		\item \label{eq:1} If for $c \in \rr$ and $\delta > 0$ we have $$\varphi^{-1}[c-\delta, c+\delta] \cap \mathrm{Cr(\varphi)} = \emptyset,$$ 
		then there exists $t_1< t_0$ and $0< \epsilon < \delta$ such that
		\begin{equation}
		\hh (t_1, \varphi^{c+\epsilon}) \subset \varphi^{c-\epsilon}.
		\end{equation}
		\item \label{eq:2} If $\varphi$ has finitely many critical points, and for $c \in \rr$ if $U$ is an open neighborhood of $\mathrm{Cr}(\varphi,c)$ ($U =\emptyset$ if $\mathrm{Cr}(\varphi,c) = \emptyset$), then there exist $t_1 < t_0$ and $\epsilon > 0$ such
		that
		\begin{equation}
		\hh (t_1, \varphi^{c+\epsilon}\setminus U) \subset \varphi^{c-\epsilon}.
		\end{equation}
	\end{enumerate}	
\end{cor}
\begin{theorem}\label{th:link}
	Let $ \fs{M} $ be a $ C^1 $-\fr manifold endowed with a complete Finsler metric $ \uprho $ and let $ \va : \fs{M} \to \rr $
	be a closed $C^1$-functional.  Suppose $ \set{S_0,S,C} $ is a linking set through $ \gamma \in \mathsf{C}(S_0,\mathsf{T}) $, $ C $ is closed and $ \uprho (\gamma(S_0),C) > 0 $. Suppose the following conditions hold
	\begin{description}
		\item[(LT1)] $ \bl{s} \coloneq \sup_{\gamma(S_0)} \leq \inf_{C} \va \eqcolon \bl{i}, $
		\item[(LT2)] $ \va $ satisfies the $\ps$-condition at 
		\begin{equation}\label{eq4}
		c \coloneq \inf_{h \in \eu{H}}\sup_{x \in S}\va (\gamma(x)),
		\end{equation}
	    where
		$$
		\eu{H} \coloneq \Set{h \in \eu{C}(S,\mathsf{T})}{ h|_{\partial S_0} = \gamma}.
		$$
	\end{description}
	Then $ c $ is a critical value and $ c \geq \bl{i} $. Furthermore, if $ c= \bl{i} $ then $\mathrm{Cr}(\varphi,c) \medcap   C \neq \emptyset$.
\end{theorem}
\begin{proof}
	Let $ h \in \eu{H} $, then by definition of linking we have $ h (S) \medcap C \neq \emptyset $ and so
	$ c \geq \bl{i} $.
	
	 At first, suppose that  $ c > \bl{i} $. If $ \mathrm{Cr}(\va,c) \medcap C = \emptyset $,
	then by Lemma~\ref{contra} there exists $\epsilon>0$ such that $\Phi_{\varphi} (x) < -\epsilon $ for all $x \in \fs{M}$. Let $ \upepsilon \coloneq \min \set{\epsilon, \bl{s} -c} $ and let 
	$$ B \coloneq \Set{ x \in \fs{M}}{\abs{\va(x) - c} \leq \upepsilon}.  $$ Since $ \va $ is closed and \cs it follows that its inverse is also closed and so $ B $ is closed. Suppose $ A $ is a compact subset of $ \fs{M} $ such that
	$ B \medcap A = \emptyset $ and for some $ k > 1 $ 
	$$
	\Phi_\varphi(x) < -2\upepsilon(1+k^2) \quad \forall x \in A.
	$$
	Then, by Lemma \ref{defor} there exist $ t_0 > 0 $ and $ B $-invariant deformation $ \hh $ for $ [0,t_0) $.
	
	 By Lemma \ref{titi}\eqref{eq:1} then we can find $ t_1 < t_0 $ and $ \varepsilon < \upepsilon $ such that 
	\begin{equation}\label{eq3}
	\hh (t_1, \varphi^{c+\varepsilon}) \subset \varphi^{c-\varepsilon}.
	\end{equation}
	Fix $ \bl{h} \in \eu{H} $ such that
	\begin{equation}\label{ineq}
	\va (\bl{h}(x)) \leq c + \varepsilon \quad \forall x \in S.
	\end{equation}
	Define the mapping
	\begin{equation}
		\upgamma = \hh (t_1, \bl{h}(x)) \quad \forall x \in S.
	\end{equation}
	Thus, $ \upgamma \in \eu{C}(S, \fs{M}) $ and if  $ x \in S_0  $, then by \eqref{ineq}
	\begin{equation}
		\va (\bl{h}(x)) = \va (\gamma(x)) \leq \bl{s} \leq c - \varepsilon.
	\end{equation}
	Therefore, as $ \hh $ is $ B $-invariant by Lemma~\ref{defor}, it follows that $ \upgamma \lvert_{S_0} = \gamma $
	whence $ \upgamma \in \eu{H} $.
	But by \eqref{eq3} we have $ \va (\upgamma(x)) \leq c -\varepsilon $ for all $ x \in S $. This contradicts
	\eqref{eq4}, and so $ \mathrm{Cr}(\va,c) \medcap C \neq \emptyset $.
	
	Now assume that $ c =\bl{i} $. From $ \uprho (\gamma(S_0),C) > 0 $ it follows that one can find a closed \nbd $ U $
	of $ \gamma(S_0) $ such that
	$$
	\uprho (U,C) > 0 \quad \text{and} \quad \uprho (\gamma(S_0),\complement U) > 0.
	$$
	
	Now suppose that $ c = \bl{i} > \bl{s} $ and $ \mathrm{Cr}(\va,c) \medcap = \emptyset $. As before, one can find
	$ \varepsilon >0, t_1>0  $ and a deformation $ \hh $ that satisfies \eqref{eq3}. Also, again we fix $ \bl{h} \in \eu{H} $ such that $ \va (\bl{h}(x)) \leq c + \varepsilon $ for all $ x \in S $, and define the mapping 
	$ \upgamma(x) = \hh (t_1, \bl{h}(x)) $ for all $ x \in S $. As above $ \upgamma \in \eu{H} $.
	From \eqref{eq3} it follows that for all  $x \in S  $ we have
	$$
	\va (\upgamma(x)) \leq c- \varepsilon< \inf_{C} \va \quad \text{or} \quad \upgamma(x) \in \complement C.
	$$
	Thus, $ \upgamma \notin C $ for all $ x \in S $ and so $ \upgamma(S) \medcap C = \emptyset $, which is a contradiction.
	
	Now suppose that $ c = \bl{i} = \bl{s} $. Let $ f : \fs{M} \to \rr $ be a mapping of class $ C^1 $ such that
	$$
	f \rvert_{\gamma(S_0)} =0 \quad \text{and} \quad f \rvert_{\complement U} =1.
	$$
We may replace $ \va(x) $ by $\va(x) + (1-\bl{s})  $ and assume that $ \bl{s} > 0 $. Define 
$$
\upvarphi (x) = (f\varphi) (x) \quad \text{and} \quad \bl{c} = \inf_{h \in \eu{H}}\sup_{x \in S}\upvarphi(h(x)).
$$	
On $ \complement S $ we have $ \varphi = \upvarphi $, therefore, $ \upvarphi $ satisfies the $ (\mathrm{PS})_{\bl{c}} $-condition on $ \complement S $ and $$ \sup_{\gamma(S_0)}\upvarphi = 0 < \bl{i} = \inf_{C} \upvarphi. $$
Thus, for $ \upvarphi $ we have $ \bl{c} = \bl{i} > \bl{s} $ and by repeating the above arguments  we obtain $$ \mathrm{Cr}(\upvarphi,\bl{c}) \medcap C \neq \emptyset.$$
Since $ \va = \upvarphi $ on $ \complement S \supset C $, it follows that 
$$
\mathrm{Cr}(\upvarphi,\bl{c}) \medcap C \subset \mathrm{Cr}(\varphi,\bl{c}) \medcap C.
$$
We will show that $ c =\bl{c} $. For $ h \in \eu{H} $, by the linking assumption we have
$$
\upvarphi (h(x)) = \va (h(x)) \geq \bl{i} = c \quad \text{for}\, x\in S\, \text{such that} \, h(x) \in C.
$$
Thereby, $ \bl{c} \geq c $. Besides, since $ f (h(x)) \in [0,1] $ and $ \sup_{E} (\va \circ h) \geq c > 0 $ it follows that
$$
\upvarphi (h(x)) \leq \sup_{S} (\va \circ h) \quad x \in S.
$$
Thus, $ \bl{c} \leq c $ by the definitions of $ c $ and $ \bl{c} $. This concludes the proof.
\end{proof}	
 An immediate corollary of Theorem \ref{th:link} is the following mountain pass theorem.
\begin{theorem}\label{th:mpt}
	Suppose that $ x_0,x_1 \in \fs{M} $ and $ x_0 \in U \opn \fs{M}, \, x_1 \notin \Cl{U} $. Let $ \va : \fs{M} \to \rr $ be a closed $C^1$-functional satisfying the following condition:
	\begin{enumerate}[label=\textbf{(M\arabic*)},ref=M\arabic*]
		\item \label{eq:mpt1} $ \max \set{\va(x_0),\va(x_1)} \leq \inf_{\partial U} \va (x) \coloneq \bl{i}$;
		\item \label{eq:mpt2} $ \va $ satisfies the $\ps$-condition at 
		\begin{equation}
		c \coloneq \inf_{h \in \mathcal{C}}\sup_{t \in [0,1]}\va (h(t)),
		\end{equation}
		where
		$$
		\mathcal{C} \coloneq \Set{h \in \eu{C}([0,1],\fs{M})}{ h(0) =x_0, \, h(1) =x_1 }.
		$$
	\end{enumerate}	
Then $ c $ is a critical value and $ c \geq \bl{i} $. If $ c= \bl{i} $ then $ \mathrm{Cr}(\va,c)  \medcap U \neq \emptyset$.
\end{theorem}
\begin{proof}
	 Define the following sets:
	$$
	S \coloneq \Set{x \in \fs{M}}{(1-t)x_0 + tx_1, \, t \in [0,1]}, \quad  S_0 \coloneq \set{x_0,x_1} \quad \text{and} \quad C= \partial U.
	$$
	Pick $ \gamma \in \eu{C}(S,\fs{M}) $ such that $ \gamma(x_0) = x_0$ and $ \gamma(x_1) = x_1$.
	The image $ \gamma(S) $ is connected. Assume that $ \gamma(S) \medcap C = \emptyset $, then $ \gamma(S) = U_1 \medcup U_2 $, where
	$$
	U_1 \coloneq \gamma(S) \medcap U, \quad U_2 \coloneq \gamma(S)  \medcap \complement \Cl{U}.
	$$ 
	But $ U_0,U_1 $ are open, this contradicts the connectness of $ \gamma(S) $. Thus, $ \set{S,S_0,C} $ is a linking set through $ \gamma $. If we apply Theorem \ref{th:link} for the linking set $ \set{S,S_0,C} $ through $  \gamma $ we conclude the proof.
\end{proof} 
\begin{remk}\label{rk:mpt}
	If the inequality in the condition \eqref{eq:mpt1} is strict $ (<) $ then $ c < \bl{i} $, which means $ \va $
	has a critical point at $ x_3 $ and $ x_3 \neq x_0, x_1. $
\end{remk}
We shall apply this theorem to generalize to \fr manifolds  global diffeomorphism theorems for \fr spaces (see \cite[Theorem 3.1]{k1} and \cite[Theorem 4.1]{k2}). The proof is almost identical to the case of \fr spaces. However, for the existence of critical points we apply the following Minimax principle which is not yet available for the case of \fr spaces.
\begin{theorem}[Theorem 3.6, \cite{k}]\label{minmax}
	Let $\fs{M}$ be a connected $C^1$- Fr\'{e}chet manifold and let $\varphi: \fs{M} \to \rr$ be a non-constant closed $C^1$-functional satisfying the (PS) 
	condition at all levels. Suppose that $\mathcal{F}$ is a deformation invariant class of subsets of $\fs{M}$ and suppose that
	\begin{equation}\label{criti}
	c=c(\varphi, \mathcal{F}) \coloneq \inf_{A \in \mathcal{F}}\sup_{x \in A}\varphi(x)
	\end{equation}
	is finite, then  $c$ is the critical value for $\varphi$.	 
\end{theorem} 
\begin{theorem}\label{th:gd}
	Let $\fs{M}, \fs{N}$ be connected $C^1$- Fr\'{e}chet manifolds endowed with complete Finsler metrics $ \updelta, \uprho $ respectively. Assume that  $\varphi : \fs{M} \to \fs{N}$ is a local diffeomorphism of class $ C^1 $. Let
	$ \eu{I} : \fs{N} \to [0,\infty] $ be a closed $ C^1 $-functional such that $ \eu{I}(x)=0 $ if and only if $ x=0 $ and 
	$ \eu{I}'(x)=0 $ if and only if $ x=0 $. If for any $ q \in \fs{N} $ the functional $ \upvarphi_q $ defined by
	$$
	\upvarphi_q(x) = \eu{I}(\va(x)-q)
	$$
	satisfies the (PS)-condition at all levels, then $ \va $ is a $ C^1 $-global diffeomorphism.
\end{theorem}
\begin{proof}
	{\bf Surjectivity:} Let $ q \in \fs{N} $ be a given point. Let $ V \subset \fs{N} $ be a \nbd of $ q $, if necessarily shrink it, so that $ \va $ is diffeomorphism on $ U \coloneq \va^{-1}(V) $. The functional $\upvarphi_q(x)$ is bounded below by 0, it is of class $ C^1 $ (since  the composition of two $ C^1 $-maps is again $ C^1 $), and it satisfies the (PS)-condition. Therefore, if in Theorem \ref{minmax} we let $$\mathcal{F} = \{ \{x\} \mid x\in U\},$$ then $ c = \inf_{x \in U} \varphi(x)$ is a critical value of $ \upvarphi_q $. Therefore, for some critical point $ p \in U$ we have $\upvarphi_q(p)=c, \upvarphi'(p) = 0 $. Then, by the chain rule
	\begin{equation}\label{eq:key}
	\upvarphi_q'(p) = \eu{I}'(\va(p) - q) \circ \va'(p)=0.
	\end{equation}
	Since $ \va' $ on $ U $ is locally invertiable, it follows that \eqref{eq:key} yields $ \eu{I}'(\va(p)-q) =0 $ and thus $ \va(p)=q $.
	
	{\bf Injectivity:} Assume that $ x_1 \neq x_2 \in \fs{M} $ and $ \va(x_1) \neq \va(x_2) = q $.
	 Let $ V \subset \fs{N} $ be a \nbd of $ q $, if necessarily shrink it, so that $ \va $ is diffeomorphism on $ U \coloneq \va^{-1}(V) $.
	
	The mapping $ \va $ is open as it is local diffeomorphism, therefore, for $ r>0 $ there is $ a_r > 0 $ such that
	\begin{equation}\label{eq:g1}
	\fs{B}_{\uprho}(q,a_r) \subset \va  \big( \fs{B}_{\updelta}(x_1,r)\big) \subset V.
	\end{equation}
	Let $ \bl{r}>0 $ be small enough such that 
	\begin{equation}\label{key}
	x_2 \notin \Cl{\fs{B}_{\uprho}(x_1,\bl{r})}.
	\end{equation}
	Consider the functional $ \upvarphi_q (\va(x) - q) $. Then
	$$
	\upvarphi_q(x_1)=\upvarphi_q(x_2) =0.
	$$ 
	For $ x \in \partial{\fs{B}_{\uprho}(x_1,\bl{r})} $ in view of \eqref{eq:g1} we obtain that
	$ \va(x) \notin  \fs{B}_{\uprho}(q,a_r)$ so $ \va (x) \neq q $ on $ \partial{\fs{B}_{\uprho}(x_1,\fs{r})} $.
	Thus, 
	\begin{equation}
	\upvarphi_q (x) > 0 = \max \set{\upvarphi_q(x_1), \upvarphi_q(x_2) } \quad \text{on} \quad \partial{\fs{B}_{\uprho}(x_1,\bl{r})}.
	\end{equation}
Thereby, all the assumptions of Theorem~\ref{th:mpt} satisfies, therefore, there exists a critical point $ c \in U$
with $ \upvarphi_q(c) = \bl{c} $  for some $ \bl{c}>0 $ (see Remark \ref{rk:mpt}). But $ \bl{c} = \eu{I}(\va(c) - q) >0 $, so
\begin{equation}\label{eq:g2}
\va (c) \neq q.
\end{equation}
By the chain rule we have $ \upvarphi_q'(c) = \eu{I}'(\va(c) - q)  \circ \va'(c) = 0$.
Since $ \va' $ is invertiable on $ U $, it follows that $ \eu{I}(\va(c) - q) =0 $. Thus, $ \va(c) = q $ which contradicts \eqref{eq:g2}.
\end{proof}
Now we prove the three critical points theorem. We shall need the following strong version of Ekeland's variational principle.
\begin{theorem}[\cite{hir}, Theorem 4.7]\label{th:evpsf}
	Let $(\ftt{M}, \mathbbm{m})$ be a complete metric space.
	Let a functional $\upphi : \ftt{M} \rightarrow (-\infty, \infty]$ be lower semi-continuous, bounded
	from below and not identical to $\infty$. Let $ \varepsilon > 0 $ be an arbitrary real number, $ m \in \ftt{M} $
	a point such that
	$$
	\upphi(m) \leq \inf_{x \in \ftt{M}} \upphi(x) + \varepsilon.
	$$
	Then for an arbitrary $ r>0 $, there exists a point $ m_r \in \ftt{M} $ such that
	\begin{enumerate}[label=\textbf{(EK\arabic*)},ref=EK\arabic*]
		\item  \label{eq:eks1} $ \upphi(m_r) \leq \upphi (m)$;
		\item  \label{eq:eks2} $\mathbbm{m} (m_r,m) < r$;
		\item  \label{eq:eks3} $\upphi(m_r) < \upphi (x) + \dfrac{\varepsilon}{r} \mathbbm{m}(m_r,x) \, \forall x \in \ftt{M}\setminus \set{m_r}.$
	\end{enumerate}	
\end{theorem}
\begin{lemma} \label{l:1}
	Let $ \fs{M} $ be a \fr manifold modeled on a \fr space $ (\fs{F},\mt) $, $ \bb $ the compact bornology on $ \fs{F} $ and  $ \va : \fs{M} \to \rr $
 a $ C^1 $-functional that satisfies the (PS)-condition at all levels. Let $ \bl{u} \in U \opn M $ be such that $ \va(\bl{u}) \leq \va (u) $ for all $ u \in U $. Then, for any $ V \opn \fs{M} $ such that $ V \subset U $ either
 \begin{enumerate}[label=\textbf{(I\arabic*)},ref=I\arabic*]
 	\item   \label{eq:tcp1} $\inf_{u \in \partial W} \va (u) > \va (\bl{u})$ for some $ W \subset V $;
 	\item  \label{eq:tcp2} or for each $ W \subset V $, $ \va $ has a local minimum at a point $ c_{W} $ such that
 	$ \va(c_{W}) = \va(\bl{u}) $.
\end{enumerate}	
\end{lemma}
\begin{proof}
	Suppose $  V \subset U $ is given and \eqref{eq:tcp1} is not valid we shall prove \eqref{eq:tcp2}.
	
	Let $ (\uu_{\bl{u}}, \uppsi_{\bl{u}}) $ be a chart at $ \bl{u} $ such that $ \uu_{\bl{u}} \medcap V \neq \emptyset $, and let $ \upvarphi \coloneq \va \circ \uppsi_{\bl{u}}^{-1} $ be the local representative of $ \va $ in this chart. Thus, for any given $ W \subset V $ such that $\bl{W} \coloneq W \medcap \uu_{\bl{u}} \neq \emptyset $ we have 
	\begin{equation}\label{e:tcpt3}
	\inf_{u \in \partial \uppsi_{\bl{u}}(\bl{W})} \upvarphi(u) = \upvarphi(\bl{u}).
	\end{equation}
	Let $ \ftt{S}, \ftt{S'} $ be distinct sets such that $ \ftt{S} \subset V,\, \ftt{S}' \subset W $ and $ \Cl{\ftt{S} \setminus \ftt{S}'} \medcap \uu_{\bl{u}} \neq \emptyset $.
	In virtue of \eqref{e:tcpt3} we can find a sequence $ (u_n)_{n \in \nn} \subset \partial \uppsi_{\bl{u}}(\bl{W})$ 
	such that 
	\begin{equation}\label{eq:tcpt4}
	\upvarphi(u_n) \leq \upvarphi(\bl{u}) + \dfrac{1}{n}.
	\end{equation}
	The restriction of $ \upvarphi $ to $ \uppsi_{\bl{u}} \big(  \Cl{\ftt{S} \setminus \ftt{S}'}  \big) $
	satisfies all the assumption of Theorem \ref{th:evpsf}, therefore, there is a sequence 
	$ (v_n)_{n \in \nn} \subset  \uppsi_{\bl{u}} \big(  \Cl{\ftt{S} \setminus \ftt{S}'}  \big) $ such that
	\begin{enumerate}[label=\textbf{(E\arabic*)},ref=E\arabic*]
		\item  \label{eq:eks4} $ \upvarphi(v_n) \leq \upvarphi (u_n)$;
		\item  \label{eq:eks5} $\mt (u_n,v_n) < \dfrac{1}{n}$;
		\item  \label{eq:eks6} $\upvarphi(v_n) < \upvarphi(x) + \dfrac{1}{n} \mt(v_n,x) \quad \forall x \in \uppsi_{\bl{u}} \big(  \Cl{\ftt{S} \setminus \ftt{S}'}  \big).$
	\end{enumerate}	
It follows from \eqref{eq:eks5}	that $ (v_n)\subset \Int \uppsi_{\bl{u}} \big(  \Cl{\ftt{S} \setminus \ftt{S}'}  \big) $
for sufficiently large $ n $. In \eqref{eq:eks6}, let $ x = v_n + tb $ for sufficiently small $ t $ and $ b \in \bl{B}_{\infty}(\zero{F}) $. By Taylor's expatiation formula  of $ \upvarphi (v_n +tb) $ about $ v_n $ (\cite[Theorem 1.4.A]{ke}), and letting
$ t \to 0 $ we obtain
\begin{equation}\label{eq:tcp7}
\snorm[B]{\upvarphi(v_n)} \leq \dfrac{1}{n} \quad \forall B \in \bb.
\end{equation}
Thus, along with \eqref{eq:eks3} and the (PS)-condition for $ \upvarphi $ there exists a subsequence of $ (v_n)_{n \in \nn} $, denoted by $ (w_n)_{n \in \nn} $, such that $ w_n \to c_{W} $ for some point $ c_W \in  \partial \uppsi_{\bl{u}}(\bl{W})$. Whence, $ \upvarphi(c_w) = \upvarphi(\bl{u}), \upvarphi'(c_W) =0 $. The chain rule then completes the proof.  	
\end{proof}
\begin{theorem}[Three Critical Points Theorem]\label{co:tcpt}
	Let $ \fs{M}$ be a connected Fr\'{e}chet manifold and $ \va : \fs{M}  \rightarrow \rr $ a closed $ C^1 $- functional satisfying the Palais-Smale condition at all levels. If $ \va $ has two minima, then $ \va $ has one more critical point.
\end{theorem}
\begin{proof}
	It follows from Theorem~\ref{th:mpt} and Lemma~\ref{l:1}.
\end{proof}
\bibliographystyle{amsplain}

\end{document}